\documentclass[11pt]{article}%
\usepackage{amsmath}
\usepackage{amsfonts}
\usepackage{amssymb}
\usepackage{graphicx}%
\providecommand{\U}[1]{\protect\rule{.1in}{.1in}}
\newtheorem{theorem}{Theorem}

\newenvironment{proof}[1][Proof]{\noindent\textbf{#1.} }{\ \rule{0.5em}{0.5em}}
\begin{document}

\title{Uniqueness of asymptotic cones of complete noncompact shrinking gradient Ricci
solitons with Ricci curvature decay\vspace{-0.15in}}
\author{Bennett Chow\thanks{Dept. of Math., UC San Diego, La Jolla, CA\ 92093}
\and Peng Lu\thanks{Dept. of Math., Univ. of Oregon, Eugene, OR 97403. P.\thinspace
L. is partially supported by a grant from Simons Foundation.}}
\date{}
\maketitle

\vspace{-0.25in}

Most of this paper follows \S 2 of Kotschwar and Wang \cite{KotschwarWang}.
Let $(\mathcal{M}^{n},\bar{g},\bar{f})$ be a complete noncompact shrinking
gradient Ricci soliton with $|\operatorname{Rc}_{\bar{g}}|(x)\rightarrow0$ as
$x\rightarrow\infty$. By Munteanu and Wang \cite{MunteanuWang5}, fixing
$p\in\mathcal{M}$, there is a constant $C$ such that $|\operatorname{Rm}%
_{\bar{g}}|(x)\leq C\bar{f}\left(  x\right)  ^{-1}\leq C\left(  d_{\bar{g}%
}(x,p)+1\right)  ^{-2}$ for $x\in\mathcal{M}$. Hence $|\nabla\bar{f}|^{2}%
=\bar{f}-R_{\bar{g}}\geq\bar{f}-\frac{a^{2}}{\bar{f}}$ for some constant $a$.
By the canonical form, for $t<1$ there exist diffeomorphisms $\varphi
_{t}:\mathcal{M}\rightarrow\mathcal{M}$, defined by $\frac{\partial}{\partial
t}\varphi_{t}\left(  x\right)  =\frac{1}{1-t}\left(  \nabla_{\bar{g}}\bar
{f}\right)  \left(  \varphi_{t}\left(  x\right)  \right)  $, $\varphi
_{0}=\operatorname{id}$, such that $g(t)=\left(  1-t\right)  \varphi_{t}%
^{\ast}\bar{g}$ solves Ricci flow and $f(x,t)\doteqdot\bar{f}\left(
\varphi_{t}\left(  x\right)  \right)  >0$ satisfies $\operatorname*{Rc}%
{}_{g(t)}+\nabla_{g(t)}^{2}f(t)-\frac{1}{2(1-t)}g(t)=0$ and\vspace{-0.09in}%
\begin{equation}
\frac{\partial f}{\partial t}(x,t)=\frac{1}{1-t}|\nabla_{\bar{g}}\bar{f}%
|^{2}\left(  \varphi_{t}\left(  x\right)  \right)  \geq\frac{1}{1-t}\left(
f(x,t)-\frac{a^{2}}{f(x,t)}\right)  .\vspace{-0.08in} \label{df dt equation}%
\end{equation}
Suppose $x$ satisfies $\bar{f}(x)\geq\frac{a}{\sqrt{1-\varepsilon^{2}}}$,
$\varepsilon>0$. By $\frac{f}{f^{2}-a^{2}}\frac{\partial f}{\partial t}%
\geq\frac{1}{1-t}$, $f(x,t)^{2}-a^{2}\geq\left(  1-t\right)  ^{-2}\left(
\bar{f}(x)^{2}-a^{2}\right)  $, so\vspace{-0.09in}%
\begin{equation}
f(x,t)\geq\left(  1-t\right)  ^{-1}(\bar{f}(x)^{2}-a^{2})^{1/2}\geq
\varepsilon\left(  1-t\right)  ^{-1}\bar{f}(x)\quad\text{for }t\in
\lbrack0,1).\vspace{-0.08in} \label{f bar phi t}%
\end{equation}
We have $|\operatorname{Rm}_{g(t)}|_{g(t)}(x)=(1-t)^{-1}\left\vert
\operatorname{Rm}_{\bar{g}}\right\vert _{\bar{g}}\left(  \varphi_{t}\left(
x\right)  \right)  \leq\frac{C}{\left(  1-t\right)  f(x,t)}\leq\frac
{C\sqrt{1-\varepsilon^{2}}}{a\varepsilon}$. By this uniform bound for
curvature, $\int_{0}^{1}\left\vert \frac{\partial}{\partial t}g\left(
x,t\right)  \right\vert _{g\left(  x,t\right)  }dt\leq\frac{C\sqrt
{1-\varepsilon^{2}}}{\varepsilon a}$, and Shi's local derivative of curvature
estimates, there exists a smooth metric $g_{1}$ on $\{\bar{f}>a\}$ such that
$g(t)$ converges to $g_{1}$ in $C^{\infty}$ on $\{\bar{f}\geq a+\varepsilon
\}$, for every $\varepsilon>0$.

Now $\frac{\partial f}{\partial t}\left(  x,t\right)  \leq\frac{1}%
{1-t}f\left(  x,t\right)  $ implies $h(x,t)\doteqdot\left(  1-t\right)
f\left(  x,t\right)  \leq\bar{f}\left(  x\right)  $. By $0\leq R_{\bar{g}%
}\left(  \varphi_{t}\left(  x\right)  \right)  \leq\frac{C\left(  1-t\right)
}{\varepsilon\bar{f}(x)}$ and\vspace{-0.1in}%
\begin{equation}
\frac{\partial h}{\partial t}(x,t)=-f\left(  x,t\right)  +|\nabla_{\bar{g}%
}\bar{f}|^{2}\left(  \varphi_{t}\left(  x\right)  \right)  =-R_{\bar{g}%
}\left(  \varphi_{t}\left(  x\right)  \right)  \quad\text{for }x\in\{\bar
{f}\geq\tfrac{a}{\sqrt{1-\varepsilon^{2}}}\}\text{ and }t\in\lbrack
0,1),\vspace{-0.09in}\label{ddt h of t}%
\end{equation}
we see that $h(t)$ converges in $C^{0}$ on $\{\bar{f}>a\}$ as $t\rightarrow1$
to a function $h_{1}$. By $(1-t)\operatorname*{Rc}{}_{g(t)}+\nabla_{g(t)}%
^{2}h(t)-\frac{1}{2}g(t)=0$ and elliptic theory, the convergence is in
$C^{\infty}$. Taking the limit of this equation as $t\rightarrow1$, we obtain
$\nabla_{g_{1}}^{2}h_{1}-\dfrac{1}{2}g_{1}=0$. Since $(1-t)^{2}R_{g(t)}%
+\left\vert \nabla h(t)\right\vert _{g(t)}^{2}=h(t)$, we have $\left\vert
\nabla h_{1}\right\vert _{g_{1}}^{2}=h_{1}$. Moreover, $\varepsilon\bar
{f}(x)\leq h_{1}(x)\leq\bar{f}\left(  x\right)  $. Since $\left\vert
\frac{\partial h}{\partial t}\right\vert \leq\frac{C\left(  1-t\right)
}{\varepsilon\bar{f}}$, we have $\left\vert h(x,t)-h_{1}\left(  x\right)
\right\vert \leq\frac{C\left(  1-t\right)  ^{2}}{\varepsilon\bar{f}(x)}$ on
$\{\bar{f}\geq\tfrac{a}{\sqrt{1-\varepsilon^{2}}}\}\times\lbrack0,1)$.

Define $\Omega=\{h_{1}>a\}\subset\mathcal{M}$. Taking $\varepsilon=\frac
{1}{\sqrt{2}}$, we get $\{\bar{f}>\sqrt{2}a\}\subset\Omega\subset\{\bar
{f}>a\}$. The function $\rho_{1}\doteqdot2\sqrt{h_{1}}$ on $\Omega$ satisfies
$\nabla_{g_{1}}^{2}\left(  \rho_{1}^{2}\right)  =2g_{1}$, $\left\vert
\nabla\rho_{1}\right\vert _{g_{1}}^{2}=1$, $\nabla^{g_{1}}(\rho_{1}^{2})$ is a
vector field generating a $1$-parameter family $\left\{  \varphi_{t}%
^{1}\right\}  _{t\in\lbrack0,\infty)}$ of homotheties of $(\Omega,g_{1})$ into
itself, the integral curves to $\nabla^{g_{1}}\rho_{1}$ are geodesics, and
there is a diffeomorphism between $\Omega$ and the product of $(2\sqrt
{a},\infty)$ and a compact manifold $\Sigma^{n-1}$ such that $g_{1}=d\rho
_{1}^{2}+\rho_{1}^{2}\tilde{g}_{1}$, where $\tilde{g}_{1}$ is a $C^{\infty}$
metric on $\Sigma$. This implies that $(\Omega,g_{1})$ extends to a regular
cone.$\vspace{-0.07in}$

\begin{theorem}
Any two asymptotic cones of a complete noncompact shrinking gradient Ricci
soliton with $|\operatorname{Rc}|(x)\rightarrow0$ as $x\rightarrow\infty$ are
isometric.$\vspace{-0.07in}$
\end{theorem}

\begin{proof}
Let $O$ be a minimum point of $\bar{f}$, so that $\varphi_{t}(O)=O$. Suppose
that a Euclidean metric cone $\operatorname{C}$ is the pointed
Gromov--Hausdorff limit of $\left(  \mathcal{M},\lambda_{i}^{-1}d_{\bar{g}%
},O\right)  $ for some sequence $\lambda_{i}\rightarrow\infty$. Since
$g(1-\lambda_{i}^{-2})=\lambda_{i}^{-2}\varphi_{1-\lambda_{i}^{-2}}^{\ast}%
\bar{g}$ converges pointwise in $C^{\infty}$ on compact subsets of $\Omega$ to
$g_{1}$, we have that $(\varphi_{1-\lambda_{i}^{-2}}(\Omega),\lambda_{i}%
^{-2}\bar{g})$ converges in the $C^{\infty}$ Cheeger--Gromov sense using the
diffeomorphisms $\varphi_{1-\lambda_{i}^{-2}}$. Since\vspace{-0.09in}%
\[
d_{\lambda_{i}^{-2}\bar{g}}(\varphi_{1-\lambda_{i}^{-2}}(x),O)=d_{(\varphi
_{1-\lambda_{i}^{-2}}^{-1})^{\ast}g(1-\lambda_{i}^{-2})}(\varphi
_{1-\lambda_{i}^{-2}}(x),\varphi_{1-\lambda_{i}^{-2}}(O))=d_{g(1-\lambda
_{i}^{-2})}(x,O)\leq Cd_{\bar{g}}(x,O),\vspace{-0.1in}%
\]
the Cheeger--Gromov convergence matches with the pointed Gromov--Hausdorff
convergence. We obtain $(\Omega,g_{1})$ is isometric to the complement of a
compact set in $\operatorname{C}$.\thinspace So $\operatorname{C}$ is
independent of the choice of $\lambda_{i}$.\vspace{-0.15in}
\end{proof}

\end{document}